\documentclass[12pt, a4paper]{amsart}

\usepackage[utf8]{inputenc}
\usepackage[english]{babel}
\usepackage[hmargin=2.3cm,vmargin=2.3cm]{geometry}

\newtheorem{theorem}{Theorem}[section]
\newtheorem{proposition}[theorem]{Proposition}
\newtheorem{lemma}[theorem]{Lemma}

\newtheorem{example}[theorem]{Example}
\newtheorem{remark}[theorem]{Remark}
\theoremstyle{definition}
\newtheorem{definition}[theorem]{Definition}

\makeatletter
\@namedef{subjclassname@2020}{%
  \textup{2020} Mathematics Subject Classification}
\makeatother
\title{Sobolev spaces of vector-valued functions}

\author{Iv{\'a}n Caama\~{n}o}
\address{IC: Departamento de An\'alisis Matem\'atico y Matem\'atica Aplicada \\
Facultad de Ciencias Matem\'aticas \\
Universidad Complutense de Madrid \\ 28040 Madrid, Spain
28040 Madrid, Spain}
\email{ivancaam@ucm.es}

\author{Jes\'us A. Jaramillo}
\address{JAJ: Instituto de Matem\'atica Interdisciplinar (IMI) and Departamento de An\'alisis Matem\'atico y  Matem\'atica Aplicada \\ Facultad de Ciencias
  Matem\'aticas \\ Universidad Complutense de Madrid \\ 28040 Madrid, Spain}
\email{jaramil@mat.ucm.es}

\author{{\'A}ngeles Prieto}
\address{AP: Departamento de An\'alisis Matem\'atico y  Matem\'atica Aplicada \\ Facultad de Ciencias
  Matem\'aticas \\ Universidad Complutense de Madrid \\ 28040 Madrid, Spain}
\email{angelin@mat.ucm.es}

\author{Alberto Ruiz de Alarc{\'o}n}
\address{ARA: Instituto de Ciencias Matem{\'a}ticas (CSIC-UAM-UC3M-UCM) and Departamento de An\'alisis Matem\'atico y  Matem\'atica Aplicada (Universidad Complutense de Madrid) \\
C/ Nicol{\'a}s Cabrera, 13-15, Campus de Cantoblanco \\ 28040 Madrid, Spain}
\email{alberto.ruiz.alarcon@icmat.es}

\thanks{Research supported in part by grant PGC2018-097286-B-I00 (Spain).}



\begin{document}

\maketitle

\begin{abstract}
We are concerned here with Sobolev-type spaces of vector-valued functions. For
an open subset  $\Omega\subset\mathbb{R}^N$ and a Banach space $V$, we compare the
classical  Sobolev space $W^{1,p}(\Omega, V)$  with the so-called
Sobolev-Reshetnyak  space $R^{1,p}(\Omega, V)$. We see that, in general,
$W^{1,p}(\Omega, V)$ is a closed subspace of $R^{1,p}(\Omega, V)$. As a main
result, we obtain that $W^{1,p}(\Omega, V)=R^{1,p}(\Omega, V)$ if, and only if,
the Banach space $V$ has the Radon-Nikodým property
\end{abstract}

\section*{Introduction}

This paper deals with first order Sobolev spaces of vector-valued functions.
For an open subset $\Omega\subset\mathbb{R}^N$ and a Banach space $V$, we will first
consider the classical Sobolev space $W^{1,p}(\Omega, V)$ of functions defined
on $\Omega$ and taking values in $V$. This space is defined using the notion of
Banach-valued weak partial derivatives in the context of  Bochner integral, much
in the same way as the usual Sobolev space of scalar-valued functions.

A different notion of Sobolev space was introduced by Reshetnyak in
\cite{Reshetnyak} for functions defined on an open subset $\Omega\subset\mathbb{R}^N$
and taking values in a metric space. Here we will consider only the case of
functions with values in a Banach space $V$. The corresponding
Sobolev-Reshetnyak space $R^{1,p}(\Omega, V)$ has been considered in \cite{H-T}
and extensively studied in \cite{HKST-paper}. This space is defined by a
``scalarization'' procedure, by composing the functions taking values in $V$
with continuous linear functionals of the dual space $V^*$ in a suitable uniform
way. It should be noted that there is a further notion of Sobolev space, in the
more general setting of functions defined on a  metric measure space
$(X, d, \mu)$ and taking values in a Banach space $V$. This is the so-called
Newtonian-Sobolev space, denoted by $N^{1,p}(X, V)$, which is defined using the
notion of upper gradients and line integrals. This space was introduced by
Heinonen et al. in \cite{HKST-paper}, combining the approaches of Shamungalingam
in \cite{Sh} and Reshetnyak in \cite{Reshetnyak}. We refer to the book
\cite{HKST-book} for an extensive and detailed study of Newtonian-Sobolev
spaces. In the case that the metric measure space $(X, d, \mu)$ is an open
subset $\Omega$ of euclidean space $\mathbb{R}^N$, it follows from Theorem 3.17 in
\cite{HKST-paper} or Theorem 7.1.20 in \cite{HKST-book} that, in fact,
$R^{1,p}(\Omega, V)=N^{1,p}(\Omega, V)$.

Our main purpose in this paper is to compare the spaces $W^{1,p}(\Omega, V)$ and
$R^{1,p}(\Omega, V)$. In general, we have that $W^{1,p}(\Omega, V)$ is a closed
subspace of $R^{1,p}(\Omega, V)$. As a main result, we obtain that
$W^{1,p}(\Omega, V)=R^{1,p}(\Omega, V)$ if, and only if, the space $V$ has the
Radon-Nikodým property (see Theorem \ref{main thm}). Note that this contradicts
Theorem 2.14 of \cite{H-T}. It turns out that the proof of Theorem 2.14 of
\cite{H-T} is not correct, and the gap is located in Lemma 2.12, since the
so-called $w^*$-partial derivatives need not be measurable, and in this case
they cannot be the weak partial derivatives.

The contents of the paper are as follows. In Section 1, we recall some basic
notions about measurability of Banach-valued functions and Bochner integral.
Section 2 is devoted to the concept of $p$-modulus of a family of curves. We
briefly review its definition and fundamental properties, which will be used
along the paper. In Section 3, the Sobolev space $W^{1,p}(\Omega, V)$ is
considered. In particular, we prove in Theorem \ref{acc} that every function
$f\in W^{1,p}(\Omega, V)$ admits a representative which is absolutely continuous
and almost everywhere differentiable along each rectifiable curve, except for a
family of curves with zero $p$-modulus. The Sobolev-Reshetnyak space
$R^{1,p}(\Omega, V)$ is considered in Section 4. We prove in Theorem \ref{Racc}
that every function $f\in R^{1,p}(\Omega, V)$ admits a representative which is
absolutely continuous along each rectifiable curve, except for a family of
curves with zero $p$-modulus. Finally, in Theorem \ref{main thm} we prove that
the equality $W^{1,p}(\Omega, V)=R^{1,p}(\Omega, V)$ provides a new characterization 
of the Radon-Nikodým property for Banach spaces. 

\section{Integration of vector-valued functions}

Along this section, $(\Omega ,\Sigma ,\mu )$ will denote a $\sigma$-finite
measure space and  $V$ a Banach space. We are going to recall first some basic
facts about measurability of Banach-valued functions. A function
$s:\Omega \to V$ is said to be a \emph{measurable simple function} if
there exist vectors $v_1,\ldots,v_m\in V$ and disjoint measurable subsets
$E_1,\dots,E_m$ of $\Omega$ such that
$$
s = \sum_{i=1}^m v_i \chi_{E_i}.
$$
A function $f:\Omega \to V$ is said to be \emph{measurable} if there
exists a sequence of measurable simple functions
$\{s_n:\Omega\to V\}_{n=1}^\infty$ that converges to $f$ almost
everywhere on $\Omega$. The Pettis measurability theorem gives the following
characterization of measurable functions
(see e.g. \cite{D-U} or \cite{HKST-book}):
\begin{theorem}[Pettis]
Consider a  $\sigma$-finite measure space $(\Omega ,\Sigma ,\mu )$ and a Banach
space $V$. A function $f:\Omega\to V$ is measurable if and only if
satisfies the following two conditions:
\begin{enumerate}
\item $f$ is weakly-measurable, i.e., for each $v^*\in V^*$, we have that
$\langle v^*,f\rangle :\Omega \to \mathbb{R}$ is measurable.
\item $f$ is essentially separable-valued, i.e., there exists $Z\subset\Omega$
with $\mu(Z) = 0$ such that $f(\Omega\setminus Z)$ is a separable subset of $V$.
\end{enumerate}
\end{theorem}

Let $\|\cdot\|$ denote the norm of $V$. Note that, if
$f:\Omega\to V$ is measurable, the scalar-valued function
$\|f\|: \Omega \to \mathbb{R}$ is also measurable. Also it can be seen
that any convergent sequence of measurable functions converges to a measurable
function.

For measurable Banach-valued functions, the Bochner integral is defined as
follows. Suppose first that  $s=\sum_{i=1}^m v_i\chi_{E_i}$ is a measurable
simple function as before, where $E_1,\ldots,E_m$ are measurable, pairwise
disjoint, and furthermore $\mu(E_i)<\infty $ for each $i\in\{1,\dots ,m\}$. We
say then that $s$ is \emph{integrable} and we define the integral of $s$ by
$$
\int_\Omega s \, d\mu :=\sum_{i=1}^m   \mu (E_i) v_i.
$$
Now consider an arbitrary measurable function $f:\Omega\to V$. We say
that $f$ is \emph{integrable} if there exists a sequence $\{ s_n\}_{n=1}^\infty$
of integrable simple functions such that
$$
\lim_{n\to\infty} \int_\Omega \|s_n-f\| \, d\mu = 0.
$$
In this case, the \emph{Bochner integral} of $f$ is defined as:
$$
\int_\Omega f \, d\mu:=\lim_{n\to\infty} \int_\Omega s_n \, d\mu.
$$
It can be seen that this limit exists as an element of $V$, and it does not
depend on the choice of the sequence $\{s_n \}_{n=1}^\infty$. Also, for a
measurable subset $E\subset \Omega$, we say that $f$ is integrable on $E$ if
$f \, \chi_E$ is integrable on $\Omega$, and we denote
$\int_E f \, d\mu =\int_\Omega f \, \chi_E \, d\mu$.
The following characterization of Bochner integrability will be useful
(see e.g. Proposition 3.2.7 in \cite{HKST-book}):

\begin{proposition}
Let $(\Omega ,\Sigma ,\mu )$ be a $\sigma$-finite measure space and $V$ a Banach
space. A function $f:\Omega \to V$ is Bochner-integrable if, and only
if, $f$ is measurable and $\int_\Omega \|f\| \, d\mu <\infty$.
\end{proposition}

Furthermore, if $f:\Omega \to V$ is integrable, then for each
$v^*\in V^*$ we have that $\langle v^*,f\rangle :\Omega \to \mathbb{R}$
is also integrable, and
$$
\left\langle v^*, \int_{\Omega} f \, d\mu \right\rangle =
\int_{\Omega} \langle v^*, f\rangle \, d\mu.
$$
In addition,
$$
\left\|\int_{\Omega} f \, d\mu \right\| \leq
\int_{\Omega} \|f\| \, d\mu.
$$

Finally, we introduce the classes of Banach-valued $p$-integrable functions on
$(\Omega ,\Sigma ,\mu )$ in the usual way. We refer the reader to \cite{D-U} or
\cite{HKST-book} for further information. Fix $1\leq p < \infty$. Then
$L^p(\Omega, V)$ is defined as the space of all equivalence classes of
measurable functions $f: \Omega \to V$ for which
$$
\int_\Omega\|f\|^p \, d\mu <\infty.
$$
Here, two measurable functions $f, g: \Omega \to V$ are \emph{equivalent} if
they coincide almost everywhere, that is,
$\mu(\{x\in \Omega:f(x)\neq g(x)\})=0$. It can be seen that the space
$L^p(\Omega, V)$ is a Banach space endowed with the natural norm
$$
\|f\|_p:=\left(\int_\Omega\|f\|^p d\mu \right)^\frac{1}{p}.
$$
As customary, for scalar-valued functions we denote
$L^p(\Omega )=L^p(\Omega ,\mathbb{R})$.

In the special case that $\Omega$ is an open subset of euclidean space
$\mathbb{R}^N$, endowed with the Lebesgue measure, we will also consider the
corresponding spaces $L^p_{\mathrm{loc}}(\Omega, V)$ of Banach-valued
\emph{locally $p$-integrable} functions. We say that a measurable function
$f: \Omega \to V$ belongs to $L^p_{\mathrm{loc}}(\Omega, V)$ if every point in
$\Omega$ has a neighborhood on which $f$ is $p$-integrable.

\section{Modulus of a family of curves}

The concept of modulus of a curve family can be defined in the general setting
of metric measure spaces (see e.g. \cite{Survey} or Chapter 5 of
\cite{HKST-book} for a detailed exposition) but we will restrict ourselves to
the case of curves defined in an open subset $\Omega$ of space
$\mathbb{R}^N$, where we consider the Lebesgue measure $\mathcal L^N$
and the euclidean norm $|\cdot|$.
By a \emph{curve} in  $\Omega$ we understand a continuous function
$\gamma:[a, b]\to \Omega$, where $[a, b]\subset\mathbb{R}$ is a compact
interval. The \emph{length} of $\gamma$ is given by
$$
\ell(\gamma):=\sup_{t_0 < \cdots < t_n }
\sum_{j=1}^n | \gamma(t_{j-1})-\gamma(t_j)|,
$$
where the supremum is taken over all finite partitions $a=t_0<\cdots<t_n=b$
of the interval $[a, b]$. We say that $\gamma$ is  \emph{rectifiable} if its
length is finite.  Every rectifiable curve $\gamma$ can be re-parametrized so
that it is \emph{arc-length parametrized}, i.e., $[a, b]=[0,\ell(\gamma)]$
and for each $0\leq s \leq t \leq \ell(\gamma)$ we have
$$
\ell(\gamma|_{[s, t]})=t-s.
$$
We can assume all rectifiable curves to be arc-length parametrized as above.
The integral of a Borel function $\rho:\Omega \to[0,\infty]$ over an
arc-length parametrized curve $\gamma$ is defined as
$$
\int_\gamma \rho \,  ds := \int_0^{\ell(\gamma)} \rho (\gamma (t)) \,  dt .
$$

In what follows, let $\mathcal{M}$ denote the family of all nonconstant
rectifiable curves in $\Omega$. For each subset $\Gamma\subset\mathcal{M}$,
we denote by $F(\Gamma )$ the so-called \emph{admissible functions} for
$\Gamma$, that is, the family of all Borel functions
$\rho :\Omega\to [0,\infty ]$ such that
$$
\int_\gamma \rho \, ds\geq 1
$$
for all $\gamma\in\Gamma$.
Then, for each $1\leq p <\infty$, the \emph{$p$-modulus of $\Gamma$} is defined
as follows:
$$
\mathrm{Mod}_p(\Gamma ):=
\inf_{\rho\in F(\Gamma )}\int_\Omega \rho^p \, d \mathcal L^N.
$$
We say that a property holds for \emph{$p$-almost every curve}
$\gamma\in\mathcal{M}$ if the $p$-modulus of the family of curves failing the
property is zero.
The basic properties of $p$-modulus are given in the next proposition (see e.g.
Theorem 5.2 of \cite{Survey} or Chapter 5 of \cite{HKST-book}):

\begin{proposition}
The $p$-modulus is an outer measure on $\mathcal{M}$, that is:
\begin{enumerate}
\item $\mathrm{Mod}_p(\emptyset)=0$.
\item If $\Gamma_1\subset\Gamma_2$ then $\text{\rm Mod}_p(\Gamma_1 )\leq
\mathrm{Mod}_p(\Gamma_2)$.
\item $\mathrm{Mod}_p \left(\bigcup_{n=1}^\infty \Gamma_n\right) \leq
\sum_{n=1}^\infty\mathrm{Mod}_p(\Gamma_n )$.
\end{enumerate}
\end{proposition}

For the next characterization of families of curves with zero $p$-mo\-du\-lus we
refer to Theorem 5.5 of \cite{Survey} or Lemma 5.2.8 of \cite{HKST-book}:

\begin{lemma}\label{inftyint}
Let $\Gamma\subset\mathcal{M}$. Then $\mathrm{Mod}_p(\Gamma )=0$ if, and only if, there
exists a nonnegative Borel function $g\in L^p(\Omega )$ such that
$$\int_\gamma g \, ds=\infty$$
for all $\gamma\in\Gamma$.
\end{lemma}

We will also use the following fact (see, e.g. Lemma 5.2.15 in \cite{HKST-book}):

\begin{lemma}\label{Gamma+}
Suppose that $E$ is a subset of $\Omega$ with zero-measure and denote
$\Gamma_E^+ :=\{ \gamma\in \mathcal{M}:
\mathcal{L}^1(\{ t\in [0,\ell (\gamma )]:\gamma (t)\in E\} )> 0\}$.
Then, for every $1\leq p<\infty$, $\mathrm{Mod}_p(\Gamma_E^+)=0$.
\end{lemma}

Next we give a relevant example concerning $p$-modulus:

\begin{lemma}\label{cubes}
Let $N > 1$ be a natural number, let
$w \in \mathbb R^N$ be a vector with $| w | =1$ and let $H$ be a
hyperplane orthogonal to $w$, on which we consider the corresponding
$(N-1)$-dimensional Lebesgue measure $\mathcal L^{N-1}$. For each Borel subset
$E\subset H$ consider the family $\Gamma (E)$ of all nontrivial straight
segments parallel to $w$ and contained in a line passing through $E$. Then, for
a fixed $1\leq p< \infty$, we have that $\mathrm{Mod}_p(\Gamma (E))=0$ if, and only if,
$\mathcal L^{N-1}(E)=0$.
\end{lemma}

\begin{proof}
Each curve in $\Gamma (E)$ is of the form $\gamma_x(t) = x+ tw $, for some
$x\in E$, and is defined on some interval $a\leq t \leq b$. For each
$q, r \in \mathbb Q$ with $q<r$, let $\Gamma_{q, r}$ denote the family of all
such paths $\gamma_x$, where $x\in E$, which are defined on the fixed interval
$[q, r]$. According to the result in 5.3.12 by \cite{HKST-book}, we have that
$$
\mathrm{Mod}_p(\Gamma_{q, r})= \frac{\mathcal L^{N-1}(E)}{(r-q)^p}.
$$
Suppose first that $\mathcal L^{N-1}(E)=0$. Then $\mathrm{Mod}_p(\Gamma_{q, r})=0$
for all $q, r \in \mathbb Q$ with $q<r$. Thus by subadditivity we have that
$\mathrm{Mod}_p(\bigcup_{q, r}\Gamma_{q, r})=0$. Now each segment
$\gamma_x \in \Gamma (E)$ contains a sub-segment in some $\Gamma_{q, r}$. This
implies that the corresponding admissible functions satisfy
$F(\bigcup_{q, r}\Gamma_{q, r}) \subset F(\Gamma (E))$, and therefore
$$
\mathrm{Mod}_p(\Gamma (E)) \leq \mathrm{Mod}_p \Big(\bigcup_{q,r}\Gamma_{q,r} \Big)=0.
$$
Conversely, if $\mathrm{Mod}_p(\Gamma (E))=0$ then $\mathrm{Mod}_p(\Gamma_{q, r})=0$
for any $q, r \in \mathbb Q$ with $q<r$, and therefore $\mathcal L^{N-1}(E)=0$.
\end{proof}

We finish this Section with the classical Fuglede's Lemma (for a proof, see e.g.
Theorem 5.7 in \cite{Survey} or Chapter 5 in \cite{HKST-book}).

\begin{lemma}[Fuglede's Lemma]\label{fuglede}
Let $\Omega$ be an open subset of $\mathbb{R}^N$ and let $\{ g_n\}_{n=1}^\infty$
be a sequence of Borel functions $g_n:\Omega\to[-\infty,\infty]$ that converges
in $L^p(\Omega )$ to some Borel function $g:\Omega\to[-\infty,\infty]$.
Then there is a subsequence $\{ g_{n_k} \}_{k=1}^\infty$ such that
$$ \lim_{k\to\infty}\int_\gamma | g_{n_k} - g | \, ds = 0$$
for $p$-almost every curve $\gamma$ in $\Omega$.
\end{lemma}

\section{Sobolev spaces $W^{1,p}(\Omega, V)$}

Let $1\leq p < \infty$, consider an open subset $\Omega$ of euclidean space
$\mathbb{R}^N$, where we consider the Lebesgue measure $\mathcal{L}^N$, and let
$V$ be a Banach space. We denote by $C^\infty_0(\Omega)$ the space of all
real-valued functions that are infinitely differentiable and have compact
support in $\Omega$. This class of functions allows us to apply the integration
by parts formula against functions in $L^p(\Omega ,V)$. In this way we can
define weak derivatives as follows. Given $f\in L^p(\Omega ,V)$ and
$i\in \{1,\ldots,N\}$, a function $f_i\in L^1_{\mathrm{loc}}(\Omega, V)$ is said
to be the \emph{$i$-th weak partial derivative} of $f$ if
$$
\int_\Omega \frac{\partial \varphi}{\partial x_i} \, f =
-\int_\Omega  \varphi \, f_i
$$
for every $\varphi\in C_0^\infty (\Omega )$. As defined, it is easy to see that
partial derivatives are unique, so we denote $f_i=\partial f/\partial x_i$.
If $f$ admits all weak partial derivatives, we define its \emph{weak gradient} as
the vector $\nabla f =(f_1,\ldots ,f_N)$, and the \emph{length} of the gradient
is
$$
|\nabla f|:=
\left(
\sum_{i=1}^N \left\|\frac{\partial f}{\partial x_i}\right\|^2
\right)^{\frac{1}{2}}.
$$
Using this, the classical first-order Sobolev spaces of Banach-valued functions
are defined as follows.

\begin{definition}
Let $1\leq p < \infty$, $\Omega$ be an open subset of $\mathbb{R}^N$ and let
$V$ be a Banach space. We define the Sobolev space $W^{1,p}(\Omega ,V)$ as the
set of all classes of functions $f\in L^p (\Omega ,V)$ that admit a weak
gradient satisfying $\partial f/\partial x_i\in L^p(\Omega ,V)$ for all
$i\in\{1,\ldots,N\}$. This space
is equipped with the natural norm
$$
\| f\|_{W^{1,p}} :=
\left(\int_{\Omega} \| f\|^p\right)^{\frac{1}{p}}
+\left(\int_{\Omega} |\nabla f|^p\right)^{\frac{1}{p}}.
$$
We denote by $W^{1,p}(\Omega )=W^{1,p}(\Omega ,\mathbb{R} )$.
\end{definition}

It can be shown that the space $W^{1,p}(\Omega ,V)$, endowed with this norm, is
a Banach space. Furthermore, the Meyers-Serrin theorem also holds in the context
of Banach-valued Sobolev functions, so in particular the space
$C^1(\Omega ,V)\cap W^{1,p}(\Omega,V)$  is dense in $W^{1,p}(\Omega ,V)$. We refer to
Theorem 4.11 in \cite{Kreuter} for a proof of this fact.

Recall that a function $f:[a,b]\to V$ is \emph{absolutely continuous} if for
each $\varepsilon >0$ there exists $\delta >0$ such that for every pairwise
disjoint intervals $[a_1,b_1],\ldots ,[a_m,b_m]\subset [a,b]$ such that
$\sum_{i=1}^m |b_i-a_i|<\delta$, we have that
$$
\sum_{i=1}^m\| f(b_i)-f(a_i)\| <\varepsilon.
$$
It is well known that every function in $W^{1,p}(\Omega ,V)$ admits a
representative which is absolutely continuous and almost everywhere
differentiable along almost every line parallel to a coordinate axis
(see Theorem 4.16 in \cite{Kreuter} or Theorem 3.2 in \cite{A-K}), where
differentiability is understood in the usual Fr\'echet sense. More generally,
we are going to show that this property can be extended to $p$-almost every
rectifiable curve on $\Omega$. We first need the following lemma:

\begin{lemma}{\label{lemma C1}}
Let $\Omega$ be an open subset of $\mathbb{R}^N$ and let $V$ be a Banach space.
If $f\in C^1(\Omega, V)$ and $\gamma$ is a rectifiable curve in $\Omega$,
parametrized by arc length, then $f\circ\gamma$ is absolutely continuous and
differentiable almost everywhere. Moreover, the derivative of $f\circ\gamma$
belongs to $L^1([0,\ell (\gamma)], V)$ and
$$
(f\circ\gamma)(t) - (f\circ\gamma)(0) = \int_0^t (f\circ\gamma)'(\tau) \, d\tau.
$$
for each $t\in [0,\ell (\gamma)]$.
\end{lemma}
\begin{proof}
Since $\gamma :[0,\ell (\gamma )]\to \Omega$ is a rectifiable curve
parametrized by arc length, in particular it is $1$-Lipschitz, so it is
differentiable almost everywhere. Furthermore, the derivative $\gamma'(\tau)$
has Euclidean norm $|\gamma '(\tau)|=1$ whenever it exists.
Additionally $f\in C^1(\Omega ,V)$, so the chain rule yields that $f\circ\gamma$
is differentiable almost everywhere. Now denote $h=f\circ\gamma $. Since
$$h^\prime (t)=\lim_{n\to\infty}\frac{h(t+1/n)-h(t)}{1/n}$$
we see that  $h^\prime$ is limit of a sequence of measurable functions, and
hence measurable. Furthermore, as $f\in C^1(\Omega ,V)$ and
$\gamma ([0,\ell(\gamma)])$ is compact, there exists $K>0$ such that
$| \nabla f (\gamma(\tau))|\leq K$ for all $\tau \in [0,\ell(\gamma)])$.
Then
\begin{align*}
\|h'\|_1
&=\int_0^{\ell(\gamma)}\|(\nabla f(\gamma(\tau)))\cdot\gamma'(\tau)\| \, d\tau
=\int_0^{\ell(\gamma)} \left\|\sum_{i=1}^N
\frac{\partial f(\gamma(\tau))}{\partial x_i} \cdot \gamma_i'(\tau)\right\| \, d\tau \\
&\leq \int_0^{\ell(\gamma)}\sum_{i=1}^N
    \left\| \frac{\partial f(\gamma(\tau))}{\partial x_i} \right\| \cdot
    \vert \gamma_i'(\tau) \vert \, d\tau
\leq \int_0^{\ell(\gamma)}\vert \nabla f(\gamma(\tau))\vert \cdot
    \vert \gamma'(\tau)\vert \, d\tau \leq K\ell(\gamma),
\end{align*}
concluding that $h^\prime\in L^1([0,\ell (\gamma )],V)$.
Now for each $v^*\in V^*$, applying the Fundamental Theorem of Calculus to the scalar function $\langle v^*, h\rangle$ we see that for each $t\in [0,\ell (\gamma)]$ we have that
$$\langle v^*,h\rangle (t)-\langle v^*,h\rangle (0)=\int_0^t \langle v^*,h^\prime (\tau)\rangle \, d\tau= \left\langle v^*, \int_0^t h'(\tau) \,  d\tau \right\rangle.$$
As a consequence, $h(t)-h(0)=\displaystyle\int_0^t h^\prime (\tau) \,  d\tau$ for every $t\in [0,\ell (\gamma)]$.
\end{proof}

\begin{theorem}\label{acc}
Let $1\leq p <\infty$, let $\Omega$ be an open subset of $\mathbb{R}^N$ and let
$V$ be a Banach space. Then every $f\in W^{1,p}(\Omega ,V)$ admits a
representative which is absolutely continuous and differentiable almost
everywhere over $p$-almost every rectifiable curve $\gamma$ in $\Omega$.
\end{theorem}

\begin{proof}
Let $\mathcal{M}$ denote the family of all nonconstant rectifiable curves in
$\Omega$ which, without loss of generality,  we can assume to be parametrized
by arc length. By the Meyers-Serrin density theorem, there exists a sequence
$\{ f_n\}_{n=1}^\infty$ of functions in $C^1(\Omega ,V)$ converging to $f$ in
$W^{1,p}(\Omega ,V)$-norm. In particular,  $f_n$ converges to $f$ in
$L^p(\Omega, V)$, and then there exists a subsequence of $\{ f_n\}_{n=1}^\infty$,
still denoted by $f_n$, converging almost everywhere to $f$. Choose a null
subset $\Omega_0\subset\Omega$ such that $f_n\to f$ pointwise on
$\Omega\setminus\Omega_0$. Now consider
$$
\Gamma_{\Omega_0}^+:=\{ \gamma:[0,\ell(\gamma)]\to \Omega \,\in \mathcal{M}:
    \mathcal{L}^1(\{ t\in [0,\ell (\gamma )]:\gamma (t)\in \Omega_0\})> 0\}.
$$
By Lemma \ref{Gamma+}, $\mathrm{Mod}_p(\Gamma_{\Omega_0}^+)=0$. In addition, for
every curve $\gamma\in\mathcal{M}\setminus \Gamma_{\Omega_0}^+$ the set
$E:= \{ t\in [0,\ell (\gamma )]:\gamma (t) \in\Omega _0\}$ has zero measure,
and therefore $f_n\circ\gamma\to f\circ\gamma$ almost everywhere on
$[0,\ell(\gamma)]$.

On the other hand, as $f_n\to f$ in $W^{1,p}(\Omega, V)$, we also have that
$|\nabla f_{n}-\nabla f |\to 0$ in $L^p(\Omega )$. Then we can apply Fuglede's
Lemma \ref{fuglede} and we obtain a subsequence of $\{ f_n\}_{n=1}^\infty$,
that we keep denoting by $f_n$, such that
\begin{equation}\label{eqfuglede}
\lim_{n\to\infty} \int_\gamma |\nabla f_n - \nabla f| \, ds = 0
\end{equation}
for every curve $\gamma\in\mathcal{M}\setminus \Gamma_1$, where
$\mathrm{Mod}_p(\Gamma_1)=0$. Notice that for every curve
$\gamma \in\mathcal{M}\setminus \Gamma_1$ the Fuglede identity
\eqref{eqfuglede} will also hold for any subcurve of $\gamma$, since
$$
\int_{\gamma|_{[s,t]}}|\nabla f_n-\nabla f| \, ds
\leq
\int_\gamma |\nabla f_n-\nabla f| \, ds
$$
for each $0\leq s\leq t\leq \ell (\gamma )$.

Furthermore, by Lemma \ref{inftyint}, the family of curves $\Gamma_2$ satisfying
that $\int_\gamma |\nabla f|ds =\infty$ or $\int_\gamma |\nabla f_n|ds =\infty$
for some $n$ has null $p$-modulus. Finally, we consider the family $\Gamma =\Gamma_1\cup\Gamma_2\cup\Gamma_{\Omega_0}^+$ and note that, by subadditivity,
$\mathrm{Mod}_p(\Gamma )=0$.

Now fix a rectifiable curve $\gamma\in\mathcal{M}\setminus\Gamma$. For each
$n\in\mathbb{N}$ by Lemma \ref{lemma C1} the function $f_{n}\circ\gamma$ is
almost everywhere differentiable, its derivative
$g_n=(f_{n}\circ\gamma)' = (\nabla f_n\circ\gamma )\cdot\gamma^\prime$ belongs
to $L^1([0,\ell (\gamma)], V)$ and satisfies
\begin{equation}\label{eqlemac1}
f_n\circ\gamma (t)-f_n\circ\gamma (s)=\int_s^t g_n \, d\mathcal{L}^1
\end{equation}
for each $s,t\in [0,\ell (\gamma )]$. Moreover, taking into account that
$\gamma$ is pa\-ra\-me\-trized by arc-length, we see that $|\gamma'|=1$ almost
everywhere on $[0,\ell (\gamma )]$, and we obtain that, for every function
$u\in W^{1,p}(\Omega ,V)$,
\begin{align*}
\|(\nabla u\circ\gamma )\cdot\gamma^\prime  \|
&=
\left\|\sum_{i=1}^N \left(\frac{\partial u}{\partial x_i}\circ\gamma\right)\cdot \gamma^\prime _i\right\|
\leq \sum_{i=1}^N \left\|\left(\frac{\partial u}{\partial x_i}\circ\gamma\right)\cdot \gamma^\prime _i\right\|\\
&=  \sum_{i=1}^N \left\| \frac{\partial u}{\partial x_i}\circ\gamma \right\| \cdot | \gamma^\prime _i |
\leq |\nabla u \circ\gamma| \cdot |\gamma'| = |\nabla u \circ\gamma|.
\end{align*}
Then for any $0\leq s\leq t\leq \ell (\gamma )$ we have that
\begin{align*}
\left\|\int_s^t g_n\, d\mathcal{L}^1 -\int_s^t(\nabla f\circ\gamma)\cdot\gamma'\, d\mathcal{L}^1\right\|
&\leq \int_s^t \| g_n -(\nabla f\circ\gamma )\cdot\gamma^\prime\| \, d\mathcal{L}^1
\\
&=\int_s^t \left \| (\nabla f_n\circ\gamma -\nabla f \circ\gamma )\cdot \gamma' \right\|  d\mathcal{L}^1 \\
&\leq \int_s^t |\nabla f_n-\nabla f|\circ\gamma \, d\mathcal{L}^1\\
&\leq\int_\gamma |\nabla f_n-\nabla f| \, ds\overset{n\to\infty}{\longrightarrow}0.
\end{align*}
Hence $(\nabla f\circ\gamma )\cdot\gamma^\prime\in L^1([0,\ell (\gamma )], V)$ and
\begin{equation}\label{convderivadas}
\lim_{n\to\infty }\int_s^t g_{n} \, d\mathcal{L}^1 =\int_s^t (\nabla f\circ\gamma )\cdot\gamma^\prime \, d\mathcal{L}^1.
\end{equation}
Next we are going to see that  the sequence $\{ f_n\circ\gamma\}_{n=1}^\infty $
is equicontinuous. This will follow from the fact  that
$\{ |\nabla f_n \circ \gamma | \}_{n=1}^\infty$ is equiintegrable, that is, for
every $\varepsilon >0$ there exists $\delta >0$ such that
$$
\sup_{n\geq 1}\int_A |\nabla f_n \circ\gamma| \, d\mathcal{L}^1\leq
\varepsilon\text{ if }A\subset [0,\ell (\gamma )] \text{ and }\mathcal{L}^1(A)<\delta.$$
Fix $\varepsilon >0$. Then by \eqref{eqfuglede} there exists $n_0\in \mathbb{N}$
such that
\begin{equation}\label{estequiintegrable1}
\int_0^{\ell (\gamma )}  | |\nabla f_n \circ\gamma |-|\nabla f\circ\gamma| |
 \, d\mathcal{L}^1 <\frac{\varepsilon}{2}\quad\forall n\geq n_0
\end{equation}
Now notice that as $\gamma\notin\Gamma_2$ then $|\nabla f_n \circ\gamma |$ and
$|\nabla f\circ\gamma|$ are integrable on $[0,\ell (\gamma )]$,  hence by the
absolutely continuity of the integral  we can choose a $\delta >0$ such that for
any $A\subset [0,\ell (\gamma )]$ with $\mathcal{L}^1(A)<\delta$
\begin{equation}\label{estequiintegrable2}
\int_A |\nabla f_n \circ\gamma | \, d\mathcal{L}^1<\frac{\varepsilon}{2},
\end{equation}
for all $n\in\{1,\ldots,n_0\}$ and
\begin{equation}\label{estequiintegrable3}
\int_A |\nabla f\circ\gamma| \,  d\mathcal{L}^1<\frac{\varepsilon}{2}.
\end{equation}
Then for $n\geq n_0$ by \eqref{estequiintegrable1} and \eqref{estequiintegrable3}
\begin{align*}
\int_A |\nabla f_n \circ \gamma|\, d\mathcal{L}^1
&\leq  \int_A |\nabla f \circ\gamma|\,  d\mathcal{L}^1 + \int_0^{\ell (\gamma )}
\left|
|\nabla f_n \circ\gamma| -|\nabla f \circ\gamma |
\right| \, d\mathcal{L}^1\\
&<\frac{\varepsilon}{2}+\frac{\varepsilon}{2}=\varepsilon.
\end{align*}
This, together with \eqref{estequiintegrable2}, gives that
$$
\int_A |\nabla f_n \circ\gamma | \, d\mathcal{L}^1<\varepsilon
$$
for every
$n\in\mathbb{N}$, as we wanted to prove. Hence by \eqref{eqlemac1} we have that,
if $0\leq s \leq t\leq \ell (\gamma )$ are such that $|s-t|<\delta$, then
$$
\| f_{n}\circ\gamma (s)-f_{n}\circ\gamma (t)\| \leq
\int_s^t |\nabla f_n \circ\gamma| \, d\mathcal{L}^1 <\varepsilon.
$$
This yields that $\{ f_n\circ\gamma \}_{n=1}^\infty$ is an equicontinuous
sequence. Since in addition $\{ f_n\circ\gamma \}_{n=1}^\infty$  converges on a
dense subset of $[0,\ell (\gamma )]$ we obtain that, in fact,
$\{f_n\circ\gamma\}_{n=1}^\infty$ converges uniformly on $[0,\ell(\gamma)]$.

Now we choose a representative of $f$ defined as follows:
\begin{equation*}
f(x)  :=
\begin{cases}
  \, \lim_{n \to \infty} f_n(x) & \text{if the limit exists},\\
   \, 0  &\text{otherwise}.
\end{cases}
\end{equation*}
With this definition we obtain that, for every curve
$\gamma\in\mathcal{M}\setminus\Gamma$ and every $t \in [0,\ell (\gamma )]$, the
sequence $\{(f_n\circ\gamma)(t))\}_{n=1}^\infty$ converges to $f\circ\gamma(t)$.
Therefore, using \eqref{eqlemac1} and \eqref{convderivadas} we see that, for
every $s, t \in [0,\ell (\gamma )]$,
\begin{align*}
(f\circ\gamma)(t)-(f\circ\gamma)(s)
&=\lim_{n\to\infty} ((f_{n}\circ\gamma)(t)-(f_{n}\circ\gamma)(s))\\
&=\lim_{n\to\infty }\int_s^t g_{n} \, d\mathcal{L}^1 =
\int_s^t (\nabla f\circ\gamma) \cdot\gamma' \, d\mathcal{L}^1.
\end{align*}
From here we deduce that  $f\circ\gamma$ is absolutely continuous and almost
everywhere differentiable on $[0,\ell (\gamma )]$.
\end{proof}

\section{Sobolev-Reshetnyak spaces $R^{1,p}(\Omega, V)$}

A different notion of Sobolev spaces was introduced by Reshetnyak in
\cite{Reshetnyak} for functions defined in an open subset of $\mathbb R^N$ and
taking values in a metric space. Here we will consider only the case of
functions with values in a Banach space. These Sobolev-Reshetnyak spaces have
been considered in \cite{HKST-paper} and \cite{H-T}. We give a definition taken
from  \cite{H-T}, which is slightly different, but equivalent, to the original
definition in \cite{Reshetnyak}.

\begin{definition}
Let $\Omega$ be an open subset of $\mathbb{R}^N$ and let $V$ be a Banach space.
Given $1\leq p < \infty$, the Sobolev-Reshetnyak space $R^{1,p}(\Omega, V)$ is
defined as the space of all classes of functions $f\in L^p(\Omega, V)$ satisfying
\begin{enumerate}
\item for every $v^*\in V^*$ such that $\|v^*\|\leq 1$,
$\langle v^*,f\rangle\in W^{1,p}(\Omega )$;
\item\label{cond2R1p} there is a nonnegative function $g\in L^p(\Omega)$ such
that the inequality $|\nabla\langle v^*, f\rangle|\leq g$ holds almost
everywhere, for all $v^*\in V^*$ satisfying $\| v^*\| \leq 1$.
\end{enumerate}
We now define the norm
$$\| f\| _{R^{1,p}}:=\| f\| _p+\inf_{g\in \mathcal{R}(f)} \| g\| _p,$$
where $\mathcal{R}(f)$ denotes the family of all nonnegative functions
$g\in L^p(\Omega )$ satisfying \eqref{cond2R1p}.
\end{definition}

It can be checked that the space $R^{1,p}(\Omega, V)$, endowed with the norm
$\| \cdot \| _{R^{1,p}}$, is a Banach space. We also note the following.

\begin{remark}\label{remark}
Let $\Omega\subset \mathbb{R}^N$ be an open set and let $V$ be a Banach space.
If $f:\Omega\to V$ is Lipschitz and has bounded support, then
$f\in R^{1,p}(\Omega ,V)$ for each $p\geq 1$.
\end{remark}

As we have mentioned, our main goal in this note is to compare Sobolev and
Sobolev-Reshetnyak spaces. We first give a general result:

\begin{theorem}\label{thm subset}
Let $\Omega$ be an open subset of $\mathbb{R}^N$ and let $V$ be a Banach space.
For $1\leq p<\infty$,  the space $W^{1,p}(\Omega, V)$ is a closed subspace of
$R^{1,p}(\Omega, V)$ and furthermore, for every $f\in W^{1,p}(\Omega, V)$, we
have
$$
\| f\|_{R^{1,p}}\leq \| f\|_{W^{1,p}} \leq \sqrt{N} \, \| f\|_{R^{1,p}}.
$$
\end{theorem}
\begin{proof}
That $W^{1,p}(\Omega, V) \subset R^{1,p}(\Omega, V)$ and
$\|f\|_{R^{1,p}}\leq \|f\|_{W^{1,p}}$ for all $f\in W^{1,p}(\Omega, V)$ was
proved in Proposition 2.3 of \cite{H-T}.

Now we will show the opposite inequality. Consider $f\in W^{1,p}(\Omega ,V)$,
let $g \in \mathcal{R}(f)$, and choose a vector $w\in \mathbb R^N$ with $| w | =1$.
Taking into account Theorem \ref{acc} and Lemma \ref{cubes} we see that, for
almost all $x\in \Omega$,  there exists the directional derivative
$$
D_w f(x)=\lim_{t\to 0} \frac{1}{t} (f(x+t w)-f(x)) \in V.
$$
For each $v^*\in V^*$ with $\| v^*\| \leq 1$ we then have, for almost al
 $x\in \Omega$,
$$
| D_w \langle v^*, f\rangle (x) | = | \nabla \langle v^*,
f\rangle (x) \cdot w | \leq | \nabla \langle v^*, f\rangle (x) |
\leq g(x).
$$
Thus, again for almost all $x\in \Omega$,
$$
\| D_w f(x) \| = \sup_{\| v^* \|\leq 1} | \langle v^*, D_w f(x)\rangle |
= \sup_{\| v^* \|\leq 1} | D_w \langle v^*, f\rangle (x) | \leq g(x).
$$
In this way we see that the weak partial derivatives of $f$ are such that
$\|(\partial f/\partial x_i) (x)\| \leq g(x)$ for
every $i\in\{1,\ldots,N\}$ and almost all $x\in \Omega$. From here, the desired
inequality follows.
Finally, from the equivalence of the norms on $W^{1,p}(\Omega ,V)$ we see that
it is a closed subspace.
\end{proof}

However, the following simple example shows that the opposite inclusion does
not hold in general.

\begin{example}\label{example}
Consider the interval $I=(0, 1)$ and let $f:I\to \ell^\infty$ be
the function given by
$$f(t)=\left\{ \frac{\sin (nt)}{n} \right\}_{n=1}^\infty $$
for all $t\in I$.
Then $f\in R^{1,p}(I, \ell^\infty)$ but $f\notin W^{1,p}(I, \ell^\infty)$.
\end{example}
\begin{proof}
Since $f$ is Lipschitz, we see from Remark \ref{remark} that
$f\in R^{1,p}(I, \ell^\infty)$ for all $1\leq p <\infty$. Suppose now that
$f\in W^{1,p}(I, \ell^\infty)$. From Theorem \ref{acc} we have that $f$ is
almost everywhere differentiable on $p$-almost every rectifiable curve in $I$.
Since, by Lemma \ref{inftyint},  the family formed by a single nontrivial
segment $[a, b] \subset I$ has positive $p$-modulus, we obtain that $f$ is
almost everywhere differentiable on $I$. But this is a contradiction, since in
fact $f$ is nowhere differentiable. Indeed, for each $t \in I$, the limit
$$
\lim_{h\to 0} \frac{1}{h}(f(t+h) - f(t))
$$
does not exist in $\ell^\infty$. This can be seen taking into account that
$f(I)$ is contained in the space $c_0$ of null sequences, which is a closed
subspace of $\ell^\infty$, while the coordinatewise limit is
$\{\cos (nt)\}_{n=1}^\infty$, which does not belong to $c_0$.
\end{proof}

Before going further, we give the following result, which parallels Theorem
\ref{acc}, and whose proof is based on Theorem 7.1.20 of \cite{HKST-book}.

\begin{theorem}\label{Racc}
Let $\Omega\subset\mathbb{R}^N$ be an open set, let $V$ be a Banach space and
suppose $1\leq p<\infty$. Then, every $f\in R^{1,p}(\Omega, V)$ admits a
representative such that, for $p$-almost every rectifiable curve $\gamma$ in
$\Omega$, the composition $f\circ\gamma$ is absolutely continuous.
\end{theorem}

\begin{proof}
Consider $f\in R^{1,p}(\Omega ,V)$. In particular, $f$ is measurable,  hence
there exists a null set $E_0\subset\Omega$ such that $f(\Omega\setminus E_0)$
is a separable subset of $V$. Then we can choose a countable set
$\{ v_i\}_{i=1}^\infty\subset V$ whose closure in $V$ contains the set
$$
f(\Omega\setminus E_0)-f(\Omega\setminus E_0)=
\{ f(x)-f(y):x,y\in\Omega\setminus E_0\}\subset V.
$$
Additionally, we can apply the Hahn-Banach theorem to select a countable set
$\{ v_i^*\}_{i=1}^\infty\subset V^*$ such that
$\langle v_i^*,v_i\rangle =\| v_i\|$ and $\| v_i^*\| =1$ for each
$i\in\mathbb{N}$. As before, let $\mathcal{M}$ denote the family of all
nonconstant rectifiable curves in $\Omega$. From Theorem \ref{acc} we obtain
that, for each $i\in\mathbb{N}$, there is a representative $f_i$ of
$\langle v^*_i, f\rangle$ in $W^{1,p}(\Omega)$ such that $f_i$ is absolutely
continuous on $p$-almost every curve $\gamma \in \mathcal{M}$.
Let $E_i$ denote the set where $f_i$ differs
from $\langle v^*_i, f\rangle$, and define $\Omega_0= \bigcup_{i} E_i \cup E_0$,
which is also a null set.
Now let $g\in \mathcal{R}(f)$ and define
$$g^*(x):=\sup_i|\nabla \langle v_i^*, f(x)\rangle |$$
We may also assume that $g$ and $g^*$ are Borel functions and $g^*(x)\leq g(x)$
for each $x\in\Omega$. In particular, $g^*\in L^p(\Omega )$.  For a curve
$\gamma :[a, b] \to \Omega$ in $\mathcal{M}$, consider the following properties:
\begin{enumerate}
\item the function $g^*$ is integrable on $\gamma $;
\item the length of $\gamma$ in $\Omega_0$ is zero, tat is, $\mathcal{L}^1(\{ t\in [a, b]:\gamma (t) \in \Omega_0\})=0$;
\item for each $i\in\mathbb{N}$ and every $a\leq s \leq t \leq b$,
$$
|f_i(\gamma (t))- f_i(\gamma(s) )|
\leq \int_s^t |\nabla \langle v_i^*, f\rangle (\gamma(\tau)) | \, d\tau
\leq \int_{\gamma |_{[s,t]}} g^* \,  ds.
$$
\end{enumerate}
By Lemma \ref{inftyint} and Lemma \ref{Gamma+}, respectively,  we have that
properties (1) and (2) are satisfied by $p$-almost every curve
$\gamma \in \mathcal{M}$. From Theorem \ref{acc} we obtain that property
(3) is also satisfied by $p$-almost every curve $\gamma \in \mathcal{M}$.
Thus the family $\Gamma$ of all  curves $\gamma \in \mathcal{M}$ satisfying
simultaneously (1), (2) and (3) represents $p$-almost every nonconstant
rectifiable curve on $\Omega$.
Now we distinguish two cases.

First, suppose that $\gamma:[a, b] \to \Omega$ is a curve in $\Gamma$ whose
endpoints satisfy $\gamma (a),\gamma (b)\notin \Omega_0$. Hence we can choose a
subsequence $\{ v_{i_j}\}_{j=1}^\infty$ converging to
$f(\gamma (b))-f(\gamma (a))$, and then
\begin{align*}
\| f(\gamma (b))-f(\gamma (a))\|
&=  \lim_{j\to\infty}\| v_{i_j}\|
 = \lim_{j\to\infty}|\langle v_{i_j}^*,v_{i_j}\rangle |\\
&\leq  \limsup_{j\to\infty}\Big( |\langle v_{i_j}^*,v_{i_j}-f(\gamma (a))+f(\gamma (b)) \rangle |+|\langle v_{i_j}^*, f(\gamma (a))-f(\gamma (b))\rangle |\Big)\\
&\leq  \limsup_{j\to\infty} \Big( \| v_{i_j}-f(\gamma (a))+f(\gamma (b))\| +|\langle v_{i_j}^*,f(\gamma (a))\rangle  -\langle v_{i_j}^*,f(\gamma (b))\rangle |\Big)\\
& = |\langle v_{i_j}^*,f(\gamma (a))\rangle  -\langle v_{i_j}^*,f(\gamma (b))\rangle |\\
&= |f_{i_j}(\gamma (a)) - f_{i_j}(\gamma (b))|
\leq\int_\gamma g^* \, ds.
\end{align*}
Suppose now that $\gamma:[a, b] \to \Omega$ is a curve in $\Gamma$ with at least one
endpoint in $\Omega_0$. In fact, we can suppose that $\gamma (a)\in \Omega_0$.
By property (2), we can choose  a sequence $\{ t_k\}_{k=1}^\infty\subset [a, b]$
converging to $a$ and such that $\gamma (t_k)\notin \Omega_0$. Then by the
previous case
$$
\| f(\gamma (t_k))-f(\gamma (t_l)) \| \leq \int_{\gamma |_{[t_k,t_l]}}g^* \, ds
$$
for any $k,l\in\mathbb{N}$, and hence, as $g^*$ is integrable on $\gamma$, then
$\{ f(\gamma (t_k))\}_{k=1}^\infty$ is convergent. Suppose now that
$\sigma :[c, d] \to \Omega$ is another curve in $\Gamma$ satisfying
$\sigma (c)=\gamma (a)$, and let $\{s_m\}_{m=1}^\infty\subset [c, d]$ be a
sequence converging to $a$ such that $\sigma(s_m)\not\in \Omega_0$ for every
$m\in\mathbb{N}$. Then
$$
\| f(\gamma (t_k))-f(\sigma (s_m))\| \leq
\int_{\sigma|_{[c,s_m]}}g^*ds +\int_{\gamma |_{[a,t_k]}} g^* \, ds
\overset{k,m\to\infty}{\longrightarrow} 0.
$$
This proves that the limit of $f(\gamma(t_k))$ as $k\to\infty$
is independent of the curve $\gamma$ and the sequence $\{ t_k\}_{k=1}^\infty$.
Now we choose a representative $f_0$ of $f$ defined in the following way:
\begin{enumerate}
\item If $x\in\Omega\setminus \Omega_0$ we set $f_0(x)=f(x)$.
\item If $x\in \Omega_0$ and there exists $\gamma :[a, b]\to \Omega$ in
$\Gamma$ such that $\gamma(a)=x$,  we set
$f_0(x)=\lim_{k\to\infty}f(\gamma(t_k))$ where
$\{t_k\}_{k=1}^\infty\subset [a, b]$ is a sequence converging to $a$ such that
$\gamma(t_k)\notin \Omega_0$ for each $k$.
\item Otherwise, we set $f_0(x)=0$.
\end{enumerate}
By definition, $f_0=f$ almost everywhere and, for every
$\gamma:[a, b] \to \Omega$ in $\Gamma$,
$$
\| f_0(\gamma (b))-f_0(\gamma (a))\|\leq
\int_\gamma g^* \, ds \leq \int_\gamma g \, ds.
$$
Furthermore, as this also holds for any subcurve of $\gamma$ by the definition
of $\Gamma$,  we also have that for every $a\leq s \leq t \leq b$
\begin{equation}\label{acbound}
\| f_0\circ \gamma (t)-f_0\circ\gamma (s)\| \leq \int_{\gamma |_{[s,t]}}g \,ds.
\end{equation}
Therefore, the integrability of $g$ on $\gamma$ gives that $f\circ\gamma$ is
absolutely continuous.
\end{proof}

Note that in the previous theorem, in contrast with Theorem \ref{acc}, for
$p$-almost every curve $\gamma$ the composition $f\circ\gamma$ is absolutely
continuous but, in general, it needs not be differentiable almost everywhere
unless the space $V$ satisfies the Radon-Nikodým Property. Recall that a Banach
space $V$ has the \emph{Radon-Nikodým Property} if  every Lipschitz
function $f:[a,b]\to V$ is differentiable almost everywhere.
Equivalently (see e.g. Theorem 5.21 of \cite{Benyamini}) $V$ has the
Radon-Nikodým Property if and only if every absolutely continuous function
$f:[a,b]\to V$ is differentiable almost everywhere. The name of this
property is due to the fact that it characterizes the validity of classical
Radon-Nikodým theorem in the case of Banach-valued measures. We refer to
\cite{D-U} for an extensive information about the Radon-Nikodým Property on
Banach spaces.

We are now ready to give our main result:

\begin{theorem}\label{main thm}
Let $\Omega$ be an open subset of $\mathbb{R}^N$, let $V$ be a Banach space and
$1\leq p<\infty$. Then $W^{1,p}(\Omega, V)=R^{1,p}(\Omega, V)$ if, and only if,
the space $V$ has the Radon-Nikodým property.
\end{theorem}

\begin{proof}
Suppose first that $V$ has the Radon-Nikodým Property. Consider
$f\in R^{1,p}(\Omega, V)$ and let $g\in \mathcal{R}(f)$. Fix a direction $e_i$
parallel to the $x_i$-axis for any $i\in\{1,\ldots,N\}$. From Theorem \ref{Racc}
we obtain a suitable representative of $f$ such that, over $p$-almost every
segment parallel to some $e_i$,  $f$ is absolutely continuous and, because of
the Radon-Nikodým Property, almost everywhere differentiable. Therefore, by
Lemma \ref{cubes} and Fubini Theorem we have that, for almost every $x\in\Omega$
and every $i\in\{1,\ldots,N\}$, there exists the directional derivative
$$
D_{e_i}f(x) = \lim_{h\to 0} \frac{f(x+h e_i)-f(x)}{h}.
$$
Note that each $D_{e_i}f$ is measurable, and that from Equation \eqref{acbound}
above it follows that $\| D_{e_i}f(x) \| \leq g(x)$ for almost every
$x\in\Omega$. Thus $D_{e_i}f \in L^p(\Omega, V)$ for each $i\in\{1,\ldots,N\}$.
In addition, for every $v^*\in V^*$ we have that $\langle v^*, D_{e_i}f\rangle$
is the weak derivative $\langle v^*, f\rangle$. Then for every
$\varphi\in C_0^\infty (\Omega )$
$$
\left\langle v^*, \int_\Omega  \varphi \, D_{e_i}f \right\rangle
= \int_\Omega \varphi \langle v^*, D_{e_i}f \rangle
= -\!\int_\Omega \frac{\partial \varphi}{\partial x_i} \langle v^*, f \rangle
= \left\langle
v^*,-\!\int_\Omega \frac{\partial \varphi}{\partial x_i}\,f
\right\rangle.
$$
Thus for every $i\in\{1,\ldots,N\}$ the directional derivative $D_{e_i}f$ is, in
fact, the $i$-th weak derivative of $f$, that is,
$\partial f/\partial x_i=D_{e_i}f\in L^p(\Omega ,V)$. It follows that
$f\in W^{1,p}(\Omega ,V)$.

For the converse, suppose that $V$ does not have the Radon-Nikodým Property.
Then there exists a Lipschitz function $h:[a, b]\to V$ which is not
differentiable almost everywhere. We may also assume that  $[a,b]\times R_0=R$
is an $N$-dimensional rectangle contained in  $\Omega$, where $R_0$ is an
$(N-1)$-dimensional rectangle. The function $f :[a,b]\times R_0 \to V$ given by
$f(x_1, x_2, \ldots,x_N)= h(x_1)$ is Lipschitz, so it admits an extension
$\tilde{f}: \Omega \to V$ which is Lipschitz and has bounded support. Then, as
noted in Remark \ref{remark},  we have that $\tilde{f} \in R^{1,p}(\Omega, V)$.
On the other hand, $\tilde{f}$ is not almost everywhere differentiable along any
horizontal segment contained in  $[a,b]\times R_0=R$. From Lemma \ref{cubes} and
Theorem \ref{acc}, we deduce that $\tilde{f} \notin W^{1,p}(\Omega, V)$.
\end{proof}

\end{document}